\documentclass[12pt,reqno]{amsart}
\usepackage[utf8]{inputenc}
\usepackage[a4paper]{geometry}
\usepackage{amstext}
\usepackage{amsthm}
\usepackage{amssymb}
\usepackage[numbers]{natbib}
\usepackage[unicode=true,
 bookmarks=false,
 breaklinks=false,pdfborder={0 0 1},backref=section,colorlinks=false]
 {hyperref}

\makeatletter
\numberwithin{equation}{section}
\numberwithin{figure}{section}

\usepackage{mathrsfs}
\usepackage{graphics}
\usepackage[dvipdf]{graphicx}
\usepackage{a4wide}
\usepackage{setspace}
\usepackage{color}
\usepackage{amsthm}
\usepackage{paralist}
\usepackage{dsfont}
\usepackage{verbatim}
\usepackage{bm}
\usepackage{chemarr}
\usepackage{float}

\newtheorem{definition}{Definition}[section]
\newtheorem{remark}[definition]{Remark}\newtheorem{lemma}[definition]{Lemma}\newtheorem{thm}[definition]{Theorem}\newtheorem{coro}[definition]{Corollary}\newtheorem{assumption}[definition]{Assumption}

\numberwithin{equation}{section}

\makeatother

\begin{document}
\global\long\def\tp{u_{+}}
 \global\long\def\tm{u_{-}}
 \global\long\def\tg{u_{\Gamma}}
 \global\long\def\mp{m_{+}}
 \global\long\def\mm{m_{-}}
 \global\long\def\mg{m_{\Gamma}}
 \global\long\def\cp{\frac{\kappa_{+}}{c_{+}}}
 \global\long\def\cm{\frac{\kappa_{-}}{c_{-}}}
 \global\long\def\cg{\frac{\kappa_{\Gamma}}{c_{\Gamma}}}
 \global\long\def\div{\mathrm{div}}
 \global\long\def\op{\Omega_{+}}
 \global\long\def\om{\Omega_{-}}
 \global\long\def\tr{\mathrm{tr}_{\Gamma}\,}

\global\long\def\ttp{\bar{u}_{+}}
 \global\long\def\ttm{\bar{u}_{-}}
 \global\long\def\ttg{\bar{u}_{\Gamma}}
 \global\long\def\tinf{u^{\infty}}

\begin{abstract}
We show global well-posedness and exponential stability of equilibria
for a general class of nonlinear dissipative bulk-interface systems.
They correspond to thermodynamically consistent gradient structure
models of bulk-interface interaction. The setting includes nonlinear
slow and fast diffusion in the bulk and nonlinear coupled diffusion
on the interface. Additional driving mechanisms can be included and
non-smooth geometries and coefficients are admissible, to some extent.
An important application are volume-surface reaction-diffusion systems
with nonlinear coupled diffusion. 
\end{abstract}

\title[Bulk-interface interaction]{Global existence, uniqueness and stability for nonlinear dissipative
bulk-interface interaction systems}

\author{K. Disser}

\address{TU Darmstadt\\
 Fachbereich Mathematik, AG Analysis\\
 Schlossgartenstr. 7, 64289 Darmstadt \\
 Germany}

\email{kdisser@mathematik.tu-darmstadt.de}


\date{\today}
\maketitle

\section{Introduction}

\label{sec:intro} 

In this paper, we consider nonlinear bulk-interface systems that arise
from thermodynamically consistent gradient structure modelling of
the interaction of dissipative dynamics on bulk domains and on and
across interfaces or surfaces. The modelling concept was introduced
in \citep{GlitzkyMielke2013}, \citep{Miel13BII} and it was shown
that it applies to many different processes (examples and references
below in Subsection \ref{subsec:Applications}). Here, the aim is
to show global well-posedness and stability, including in particular
slow and fast diffusion in the bulk, coupled nonlinear diffusion on
the interface and general reaction-type nonlinearities on the interface
and as boundary conditions. The results are based on a specific functional
analytic framework that uses maximal parabolic regularity in $W^{-1,q}$
and global $L^{\infty}$-bounds that derive from the dissipative structure.

\subsection{Model equations\label{sub:PDE} }

The bulk $\Omega\subset\mathbb{R}^{d}$, $d=2,3$, is a bounded set
with boundary $\partial\Omega$ and it can be separated into two open
disjoint domains $\op$ and $\om$ by an interface $\Gamma$ of dimension
$d-1$. The domains $\Omega_{+},\Omega_{-}$ and manifold $\Gamma$
have unit outer normal vector fields $\nu_{+},\nu_{-},\nu_{\Gamma}$
at their boundaries. The scalar quantities $u_{+}:(0,T)\times\Omega_{+}\to\mathbb{R}$,
$u_{-}:(0,T)\times\Omega_{-}\to\mathbb{R}$ and $u_{\Gamma}:(0,T)\times\Gamma\to\mathbb{R}$
interact on and across $\Gamma$, and they satisfy the three sets
of equations
\begin{equation}
\left\{ \!\!\!\begin{array}{rcll}
\dot{u}_{+}\!-\!\div(k_{+}(u_{+})\nabla u_{+}) & = & f_{+}(u_{+}), & \text{in }(0,T)\times\Omega_{+},\\
(k_{+}(\tp)\nabla u_{+})\nu_{+}\!+\!m_{+}(u)(u_{+}\!-\!u_{\Gamma})\!+\!m_{\Gamma}(u)(u_{+}\!-\!u_{-}) & = & g_{+}(u), & \text{on }(0,T)\times\Gamma,\\
(k_{+}(\tp)\nabla u_{+})\nu_{+} & = & h_{+}(u_{+}), & \text{on }(0,T)\times\{\partial\op\backslash\Gamma\},
\end{array}\right.\label{eq:plus}
\end{equation}
and,
\begin{equation}
\left\{ \!\!\!\begin{array}{rcll}
\dot{u}_{-}\!-\!\div(k_{-}(\tm)\nabla u_{-}) & = & f_{-}(u_{-}), & \text{in }(0,T)\times\om,\\
(k_{-}(\tm)\nabla u_{-})\nu_{-}\!+\!m_{-}(u)(u_{-}\!-\!u_{\Gamma})\!+\!m_{\Gamma}(u)(u_{-}\!-\!u_{+}) & = & g_{-}(u), & \text{on }(0,T)\times\Gamma,\\
(k_{-}(\tm)\nabla u_{-})\nu_{-} & = & h_{-}(u_{-}), & \text{on }(0,T)\times\{\partial\om\backslash\Gamma\},
\end{array}\right.\label{eq:minus}
\end{equation}
and, 
\begin{equation}
\left\{ \!\!\!\begin{array}{rcll}
\dot{u}_{\Gamma}\!-\!\div_{\Gamma}(k_{\Gamma}(u)\nabla_{\Gamma}u_{\Gamma})\!-\!m_{+}(u)(u_{+}\!-\!u_{\Gamma})\!-\!m_{-}(u)(u_{-}\!-\!u_{\Gamma}) & = & f_{\Gamma}(u), & \text{in }(0,T)\times\Gamma,\\
(k_{\Gamma}(u)\nabla_{\Gamma}u_{\Gamma})\nu_{\Gamma} & = & h_{\Gamma}(u_{\Gamma}), & \text{on }(0,T)\times\partial\Gamma.
\end{array}\right.\label{eq:Gamma}
\end{equation}
To improve the presentation, we write 
\[
u=(u_{+},u_{-},u_{\Gamma}),\;f=(f_{+},f_{-},f_{\Gamma}),\dots
\]
The coefficient matrices $k$, the scalar transmission coefficients
$m$ and the external forces and inhomogeneous boundary conditions
$f,g,h$ may have (locally Lipschitz) dependence of the solution $u$
and continous dependence of the space variables with 
\begin{equation}
k_{\pm}(x,u_{\pm})\in\mathbb{R}_{\geq0}^{d\times d}\qquad\text{and }\qquad k_{\Gamma}(y,u)\in\mathbb{R}_{\geq0}^{(d-1)\times(d-1)},\label{eq:K}
\end{equation}
and 
\begin{equation}
m_{\Gamma}(y,u),m_{+}(y,u),m_{-}(y,u)\in\mathbb{R}_{\geq0}.\label{eq:M}
\end{equation}
More precise assumptions on $k,m,f,g,h$ are given in Section \ref{sec:Ass}.
Examples are slow or fast diffusion $k_{+}(u_{+})=\kappa_{0}u_{+}^{\rho-1}$
with constants $\rho\in\mathbb{R}$ and $\kappa_{0}>0$, where the
case $k_{+}(u_{+})=\frac{1}{u_{+}^{2}}$ is motivated by the entropic
structure for the system in \citep{Miel13BII} (see Subsection \ref{sub:Onsager-Model}).
The diffusion coefficient $k_{\Gamma}$ can depend on all three unknowns
$u$ in a nonlinear way and the same ist true for the transmission
coefficients $m$. An example is $m_{+}(u)=\frac{1}{u_{+}^{2}u_{\Gamma}}$
(motivated in Subsection \ref{sub:Onsager-Model}). For the global
equilibration result on all three components, it is needed that $m_{\pm}$,
$m_{\Gamma}$ are not only non-negative, but that at least two of
them are positive, so every pair of components interacts at least
indirectly across the interface.

\subsection{Applications\label{subsec:Applications}}

The system \eqref{eq:plus} – \eqref{eq:Gamma} can be related to
several types of applications:
\begin{itemize}
\item An important example is that $f,g$ are given by chemical reaction
rates for the densities $u_{+},u_{-}$ and $u_{\Gamma}$ of species
that interact and react at the interface. A concrete example is 
\begin{equation}
f_{+}(u_{+})=0,g_{+}(u)=-k\alpha(u_{+}^{\alpha}-\kappa u_{\Gamma}^{\beta}),\quad f_{\Gamma}(u)=k\beta(u_{+}^{\alpha}-\kappa u_{\Gamma}^{\beta})\label{eq:reactionRates}
\end{equation}
with $k,\kappa>0$ and $\alpha,\beta\geq1$. This situation is typical
for many (cell) biological, ecological and technological processes
\citep{beRoRo13fastRD,GlitzkyMielke2013,Keil13,Miel13BII} and these
systems have gained increasing attention in analysis and numerical
analysis \citep{BSKM2017,BFL18vsrd,FRT16vsrdNumerik,FKMMN2016BulkSurface,HausbergRoeger18bsrd,madzvamuseChung15bsrdNumerik,SharmaMorgan17bsrd}.
It is shown in \citep{Disser19} how the main result here, Theorem
\ref{thm:local}, applies in this case and allows to extend some of
the previous results in \citep{BSKM2017,BFL18vsrd,FKMMN2016BulkSurface,HausbergRoeger18bsrd,SharmaMorgan17bsrd}
to the case of nonlinear and nonlinearly coupled diffusion and to
the case that only part of the boundary of $\Omega$ is interactive.
Nonlinear and coupled diffusion often appear in the modelling of reaction-diffusion
systems, and they can be naturally associated to bulk-interface systems
\citep{beRoRo13fastRD,Bothe15Multiphysics,GlitzkyMielke2013,KjeBed08NETD,Miel13BII}.
It is the aim of future work to use the analysis started here and
in \citep{Disser19} to study more complex systems of multicomponent
reaction and diffusion.
\item A second example is the modelling of heat conduction within a bulk
material that is separated into two parts by a thin heat-conducting
plate. Particularly at high and low temperatures, thermal conductivity
of bulk and plate materials become nonlinear in their dependence of
temperature and non-equilibrium 
modeling of heat conduction across the plate leads to nonlinear transmission
coefficients $m$ of the type above, see for example \citep{ThermoMatter1970}
for material parameters and Subsection \ref{sub:Onsager-Model} on
the associated Onsager structure. In technological applications, the
geometry often includes sharp edges and singularites where interface
and boundary meet. These non-smooth settings are essentially included
in the analysis here. To the author's knowledge, it is the first result on this particular quasilinear transmission-type problem in a non-smooth setting. 
\item For models of diffusion and transport of electrical charges in semiconductor
devices, like solar cells, active interfaces often play a crucial
role, cf.\,\citep{GlitzkySolar2012}. In particular for the case
of three spatial dimensions, well-posedness for these systems in non-smooth
geometric settings is still hard to achieve but highly relevant in
modeling, simulation and optimization \citep{DisserRehberg2019vR}. In future work, the method developed here can be adapted to provide well-posedness for systems that include coupled diffusion and transport along active interfaces. 
\item The quasi-linear structure of \eqref{eq:plus} – \eqref{eq:Gamma}
may appear after a change of coordinates that transforms a free boundary
problem to a fixed domain, \citep{PruessSimonett2016}. 
\item Here, global existence and uniqueness is shown including semilinear
$f,g,h$ as long as boundedness is preserved. This includes, for example,
driving mechanisms of Allen-Cahn-type associated to phase separation,
cf.\,Corollary \ref{rhsmp} and the example in \eqref{eq:AllenCahnEx}. The result here applies if phase-separation occurs both in a bulk and along a (lower-dimensional) surface part of the material, coupled by (non-linear) transmission terms. Well-posedness for Allen-Cahn equations with dynamic interface conditions was shown in \citep{ColliSprekels2015}.
\end{itemize}

\subsection{Functional analytic setting }

We use that system \eqref{eq:plus} – \eqref{eq:Gamma} can be recast
as a quasilinear abstract Cauchy problem of the form 
\begin{align*}
\dot{u}(t)+\mathcal{A}_{u(t)}u(t) & =\mathcal{F}(u(t)),\\
u(0) & =u^{0},
\end{align*}
in the space $\mathrm{W}_0^{-1,q,q_{\Gamma}}$ of functionals on 
\[
\mathrm{W}^{1,q,q_{\Gamma}}=W^{1,q'}(\Omega_{+})\times W^{1,q'}(\Omega_{-})\times W^{1,q'_{\Gamma}}(\Gamma)
\]
with $q>d$, $q_{\Gamma}>2$ in spatial dimensions $d=2,3$, where
$\mathcal{A}_{u(t)}$ is a second-order divergence-form elliptic operator
(details in Subsection \ref{operator}). The main result is that this
problem is globally well-posed. This may seem like a special choice
of spaces at first, but actually, there is a good reason for this
choice in that the quasilinear map 
\begin{equation}\label{Lip}
\mathrm{W}^{1,q,q_{\Gamma}}\ni w\mapsto\mathrm{L}^{\infty,\infty}\ni w\mapsto k(w)\in\mathrm{L}^{\infty,\infty}\mapsto\mathcal{A}_{w}\in\mathcal{B}(\mathrm{W}^{1,q,q_{\Gamma}},\mathrm{W}_0^{-1,q,q_{\Gamma}})
\end{equation}
is well-defined, locally Lipschitz, and compact (estimate \eqref{eq:LS} and
Lemma \ref{Aiso}). In the usual weak setting with $w\in\mathrm{W}^{1,2,2}$,
without the first embedding, the second mapping is not Lipschitz,
so local well-posedness is less clear. On the other hand, if $\mathrm{W}_0^{-1,q,q}$
is replaced by a smaller space like $\mathrm{L}^{2,2}$ (the strong
setting), the third mapping is not bounded, so more than global $L^{\infty}$-estimates
are needed for global existence. \\
The $\mathrm{W}_0^{-1,q,q_{\Gamma}}$-setting is also well adapted to
non-smooth situations like non-smooth coefficients, non-smooth boundaries
and mixed boundary conditions. In \citep{DisKaiReh15OSR}, there is
a survey on results on the corresponding elliptic isomorphy that is
used for the quasilinear theory. Here, this means that it is not necessary
that all of the surface of the bulk domain is active and that local
well-posedness in this setting extends to mixed Dirichlet and Neumann
boundary conditions \citep{DisDIMX15} and that $\Omega_{+}$ and
$\Omega_{-}$ and the coefficients $k,m$ are allowed to be (spatially)
non-smooth to some extent. This is natural from a modelling point
of view, as the separation of a smooth domain by an interface, even
a plane, will usually create a kink. It can also be highly relevant
in semiconductor device modelling \citep{DisserRehberg2019vR} and
for catalysis modelling \citep{BSKM2017}.  A general advantage of
the quasilinear approach to \eqref{eq:plus} – \eqref{eq:Gamma} is
that it has a very good perturbation theory. It is a well-studied
problem how to include lower-order terms, time-dependence of coefficients
and external forces, cf.\,Subsection\,\ref{sub:Other-remarks}. For example, in \citep{DenkPrussZacher08}, local well-posedness of similar systems
is shown in an $L^{p}$-setting with smooth coefficients and geometries. \\
Here, in the proof of global well-posedness, a Schaefer fixed point argument is used to prove existence of solutions, and the local Lipschitzianity of the map \eqref{Lip} is used afterwards, to show uniqueness. This method needs an explanation as a maximal regularity approach with Lipschitzian nonlinearity is already taylored to a contraction mapping argument that provides both existence and uniqueness. The point here is that we are interested in global existence. The usual quasilinear maximal parabolic regularity approach provides well-posedness up to possible blow-up in the time-trace norm (the space of local semiflow). In the present weak setting, the trace space $\mathrm{X}^r_{q,q_\Gamma}$ of solutions $u \in C^0(J_T,\mathrm{X}^r_{q,q_\Gamma})$ embeds into the space of H\"older-continuous functions, $\mathrm{X}^r_{q,q_\Gamma} \hookrightarrow \mathrm{C}^{\beta,\beta_\Gamma}$ (Lemma \ref{lem:Cbb}), so for global existence, with the usual arguments,  (at least) global H\"older bounds would be needed. 
The method presented here circumvents the need for uniform higher-order a priori bounds by splitting existence and uniqueness proofs. To apply Schaefer's fixed point result, it is sufficient to know \emph{a posteriori} $L^\infty$-bounds, rather than \emph{a priori} bounds. This appears to be a technical simplification only, but it is a standard method used for quasilinear elliptic problems \citep{GilbargTrudinger}, so it may in principle be rewarding for parabolic systems as well.  \\
Regarding the nonlinearities in $f,g,h$ in the system  \eqref{eq:plus} – \eqref{eq:Gamma}, the need for global $L^{\infty}$-estimates can be a severe restriction, in particular, 
on reaction-diffusion problems. On the other hand, it is not a stronger
condition than what is used in the semilinear theory \citep{BFL18vsrd,FKMMN2016BulkSurface,SharmaMorgan17bsrd},
so the main result here can be applied in these situations. The boundedness-by-entropy
method developed by Jüngel \citep{Juengel15BBE} shows that there
is a connection between the entropic gradient structures of a system and
$L^{\infty}$-estimates. A direct argument for this in a simple nonlinear
diffusion-reaction case was also made in \citep{Disser19}. The proof
of $L^{\infty}$-bounds, equilibration and exponential rates for system
\eqref{eq:plus} – \eqref{eq:Gamma}, Lemma \ref{mp} and Theorem
\ref{thm:convToEq}, make exact use of the gradient structure of the
bulk-interface interaction in that for convex energies, the linearization
of the dual dissipation potentials acts like a discrete gradient on
the components. \\


\subsection*{Outline.}

The paper is organized as follows. The next section contains basic
assumptions on the geometry and coefficient functions in \eqref{eq:plus}
– \eqref{eq:Gamma} and collects preliminary results on the bilinear
form and linearized operator associated to the system. In Section
\ref{sec:The-Sesquilinear-form}, the main result on existence and
uniqueness of global solutions is proved. In Section \ref{stability},
the bulk-interface Poincaré inequality and exponential stability for
the global equilibrium under mass conservation are shown. In the last
Section \ref{sec:Extensions-and-concluding}, the relation of the
model to the entropic Onsager system of heat diffusion and transfer
derived in \citep{Miel13BII} are discussed and extensions of the
main results like the case of $\om=\emptyset$, higher regularity,
dependence of coefficients on time, and the inclusion of lower-order
perturbations are given. 

\section{Basic assumptions and functional analytic framework \label{sec:Ass}}

\subsection{Assumptions on geometry and coefficients\label{sub:AssOmega}}

\begin{assumption}\label{ass1} The bulk domains $\op$ and $\om$
are bounded Lipschitz domains \citep[Def. 1.2.12]{grisvard85}. The
interface $\Gamma$ is a $d-1$-dimensional $C^{1}$-part of the boundary
of both $\Omega_{+}$ and $\Omega_{-}$ with Lipschitz boundary $\partial\Gamma$
if $d=3$. 
\end{assumption}

For $q\in[1,\infty]$, $L^{q}(\omega)$ denotes the usual real Lebesgue
space of $q$-integrable functions on a domain or manifold $\omega$,
$W^{m,q}(\omega)$ denote the usual $L^{q}$-Sobolev spaces of order
$m\in\mathbb{N}$ and $C^{\alpha}(\omega)$ are the uniform Hölder
spaces of exponent $\alpha\geq0$ with $C^{0}(\omega)=C(\overline{\omega})$
if $\omega$ is bounded. Function spaces related to \eqref{eq:plus}
– \eqref{eq:Gamma} are: for $q,q_{\Gamma}\in[1,\infty]$, $\alpha,\alpha_{\Gamma}\geq0$
and $\mathcal{H}_{d-1}$ the $d-1$-dimensional Hausdorff measure
on $\Gamma$, 
\[
\mathrm{L}^{q,q_{\Gamma}}:=L^{q}(\op)\times L^{q}(\om)\times L^{q_{\Gamma}}(\Gamma),
\]
\[
\mathrm{W}^{1,q,q_{\Gamma}}:=W^{1,q}(\op)\times W^{1,q}(\om)\times W^{1,q_{\Gamma}}(\Gamma),\quad\text{and}
\]
\[
\mathrm{C}^{\alpha,\alpha_{\Gamma}}:=C^{\alpha}(\op)\times C^{\alpha}(\om)\times C^{\alpha_{\Gamma}}(\Gamma),
\]
where $W^{1,q}(\Gamma)$ is defined through the standard notion of
weak differentiability on the manifold $\Gamma$ that preserves embedding
and trace theorems, integration by parts and Poincaré inequalities.

The trace operator 
\begin{equation}
\mathrm{tr}_{\Gamma}:W^{1,q}(\op)\to L^{q}(\Gamma)\label{eq:trace operator}
\end{equation}
is well-defined and continuous (likewise for $\om$). Set
\begin{equation}
\mathrm{tr}_{\Gamma}\,u=(\tr u_{+},\tr u_{-},u_{\Gamma})\label{eq:trace}
\end{equation}
for the trace components of $u$ on the interface $\Gamma$. Often
we omit the operator $\tr$ (like in the statement of the model equations
\eqref{eq:Gamma} on $\Gamma$). Dual Sobolev spaces are
denoted by
\[
W_0^{-1,q}(\omega):=(W^{1,q'}(\omega))^{'}\;\text{and }\;\mathrm{W}_0^{-1,q,q_{\Gamma}}:=(\mathrm{W}^{1,q',q'_{\Gamma}})^{'}
\]
with $\frac{1}{q}+\frac{1}{q'}=\frac{1}{q_{\Gamma}}+\frac{1}{q'_{\Gamma}}=1$.
For $-\infty<l\leq L<+\infty$ and $n\in\mathbb{N}$, let 
\[
(\mathbb{R}_{l}^{L})^{n}:=\{v\in\mathbb{R}^{n}:l\leq v_{i}\leq L\,\text{ for }i=1,\dots,n\},\quad\text{and}
\]
\[
\mathrm{C}_{l}^{L}=\{u\in\mathrm{C}^{0,0}:l\leq u_{\pm}(x),\tg(y)\leq L\,\text{ for all }x\in\Omega_{\pm},y\in\Gamma\}.
\]

\begin{assumption}\label{Akm}Let $k$ and $m$ be given as in (\ref{eq:K})
and (\ref{eq:M}) and let $-\infty<l<L<+\infty$ be given constants. 
\begin{enumerate}
\item \label{akm1} Uniformly in $u\in(\mathbb{R}_{l}^{L})^{3}$, the coefficient
matrices $k(\cdot,u)$ are measurable, bounded and elliptic, i.e.
there are constants $\underline{k},\overline{k}>0$ such that 
\begin{equation}
\Vert k(\cdot,u)\Vert_{\mathrm{L}^{\infty}}\leq\overline{k},\label{bdd}
\end{equation}
and such that for all $x\in\mathbb{R}^{d}$, $y\in\mathbb{R}^{d-1}$,
\begin{equation}
x\cdot k_{\pm}(\cdot,u)x\geq\underline{k}|x|^{2}\quad\text{and}\quad y\cdot k_{\Gamma}(\cdot,u)y\geq\underline{k}|y|^{2},\label{ell}
\end{equation}
almost everywhere in $\Omega_{\pm}$, $\Gamma$. In particular, $\underline{k}$
and $\overline{k}$ may depend on $l,L$, but not on $u\in(\mathbb{R}_{l}^{L})^{3}$.
\item \label{akm2} Uniformly in $u\in(\mathbb{R}_{l}^{L})^{3}$, the transmission
coefficients $m_{\pm}$,$m_{\Gamma}$ are measurable and there are
constants $\underline{m},\overline{m}>0$ such that 
\begin{equation}
\Vert m(\cdot,u)\Vert_{L^{\infty}(\Gamma)}\leq\overline{m}\label{mbdd}
\end{equation}
and such that at least two of the three transmission functions, e.g.
$m_{+},m_{\Gamma}$ are positively bounded from below, 
\begin{equation}
\underline{m}\leq m_{+}(\cdot,u),m_{\Gamma}(\cdot,u),\label{mell}
\end{equation}
and the third transmission function is non-negative, 
\[
0\leq m_{-}(\cdot,u),
\]
almost everywhere in $\Gamma$. Note that again, $\underline{m},\overline{m}$
may depend on $l,L$, but not on $u\in(\mathbb{R}_{l}^{L})^{3}$.
\item \label{akm3} The functions $\mathbb{R}\ni u_{\pm}\mapsto k_{\pm}(x,u_{\pm})$,
$\mathbb{R}^{3}\ni u\mapsto k_{\Gamma}(y,u)$ and $\mathbb{R}^{3}\ni u\mapsto m(y,u)$
are locally Lipschitz uniformly in $y\in\Gamma,x\in\Omega_{\pm}$. 
\item \label{akm4} If $d=3$, then $k_{\pm}$ are of the form $k_{\pm}(x,u_{\pm})=\kappa_{\pm}(x,u_{\pm})\varkappa_{\pm}(x)$.
The functions $\kappa_{\pm}\colon\Omega_{\pm}\times\mathbb{R}\to\mathbb{R}$
are scalar, satisfy \ref{Akm}\eqref{akm3} and for all $u_{\pm}\in\mathbb{R}$,
we have $\kappa_{\pm}(\cdot,u_{\pm})\in C^{0}(\Omega_{\pm})$ with
$\underline{k}\leq\kappa_{\pm}(\cdot,u_{\pm})$. 
The functions $\varkappa_{\pm}\colon\Omega_{\pm}\to\mathbb{R}_{\geq0}^{3\times3}$
satisfy \ref{Akm}\eqref{akm1} and are uniformly continuous on $\Omega_{\pm}$. 
\end{enumerate}
\end{assumption} 
 Some examples of coefficients $k,m$ that satisfy Assumption \ref{Akm}
are in Subsections \ref{sub:Onsager-Model} and \ref{sub:PDE}. 
Assumption \ref{Akm}\eqref{akm4} is used to guarantee that if $d=3$, the non-autonomous operator $\mathcal{A}_{u(t)}$ that defines the system, does not change its domain of definition, see Lemma \ref{Aiso}. Note that 
the condition $\varkappa_{\pm}\in C(\overline{\Omega}_{\pm})^{3\times3}$
may be relaxed considerably,
e.g. to hold only piecewise on layers. 
For a detailed discussion of necessary and sufficient conditions for
this property, we refer to \citep{DisKaiReh15OSR}. 

\subsection{Bilinear form and elliptic operator \label{operator} }

The dissipation in \eqref{eq:plus} – \eqref{eq:Gamma} across $\Gamma$
is governed by the transmission coefficient matrix $\mathbf{m}$ given
by 
\[
\mathbf{m}=\left(\begin{array}{ccc}
\mp+\mg & -\mg & -\mp\\
-\mg & \mm+\mg & -\mm\\
-\mp & -\mm & \mp+\mm
\end{array}\right).
\]
With Assumption \ref{Akm}\eqref{akm2}, $\mathbf{m}$ is positive
semi-definite and for $r=(r_{+},r_{-},r_{\Gamma})\in\mathbb{R}^{3}$,
\begin{equation}
r\cdot\mathbf{m}r=0\;\text{a.e., if and only if }\;r_{+}=r_{-}=r_{\Gamma}.\label{eq:kerMu}
\end{equation}
Let $-\infty<l\leq L<+\infty$. For fixed $u\in\mathrm{C}_{l}^{L}$,
define the bilinear form 
\[
\mathfrak{a}_{u}:\mathrm{W}^{1,2,2}\times\mathrm{W}^{1,2,2}\to\mathbb{R}
\]
by 
\[
\mathfrak{a}_{u}(\psi,\varphi):=\mathfrak{l}_{u}(\psi,\varphi)+\mathfrak{m}_{u}(\psi,\varphi),
\]
where 
\[
\begin{array}{rcl}
\mathfrak{l}_{u}(\psi,\varphi) & := & \int_{\op}\nabla\psi_{+}\cdot k_{+}(u_{+})\nabla\varphi_{+}\,\mathrm{d}x+\int_{\om}\nabla\psi_{-}\cdot k_{-}(u_{-})\nabla\varphi_{-}\,\mathrm{d}x\\[2mm]
 &  & +\int_{\Gamma}\nabla_{\Gamma}\psi_{\Gamma}\cdot k_{\Gamma}(u)\nabla_{\Gamma}\varphi_{\Gamma}\,\mathrm{d}\mathcal{H}_{d-1},\\[2mm]
 & =: & \mathfrak{l}_{u,+}(\psi_{+},\varphi_{+})+\mathfrak{l}_{u,-}(\psi_{-},\varphi_{-})+\mathfrak{l}_{u,\Gamma}(\psi_{\Gamma},\varphi_{\Gamma}),
\end{array}
\]
and 
\[
\mathfrak{m}_{u}(\psi,\varphi)=\int_{\Gamma}\tr\psi\cdot\mathbf{m}\tr\varphi\,\mathrm{d}\mathcal{H}_{d-1}.
\]
By (\ref{eq:trace operator}), the form $\mathfrak{a}_{u}$ is well-defined
and continuous. Due to (\ref{eq:kerMu}) and Assumption \ref{Akm},
\begin{equation}
\mathfrak{a}_{u}(\varphi,\varphi)\geq0\quad\text{and}\quad\mathfrak{a}_{u}(\varphi,\varphi)=0\;\text{if and only if }\;\varphi_{+}=\varphi_{-}=\varphi_{\Gamma}\equiv\mathrm{const}.\label{eq:kera}
\end{equation}
The form $\mathfrak{a}_{u}$ induces an operator $\mathcal{A}_{u}:\mathrm{W}^{1,2,2}\to\mathrm{W}_0^{-1,2,2}$
by 
\[
\mathcal{A}_{u}(\psi)(\varphi):=\mathfrak{a}_{u}(\psi,\varphi),\quad\text{for all }\psi,\varphi\in W^{1,2,2}.
\]
For $q,q_{\Gamma}\in[2,\infty)$, let $\mathcal{A}_{u}^{q,q_{\Gamma}}$
be the maximal restriction of $\mathcal{A}_{u}$
to $\mathrm{W}_0^{-1,q,q_{\Gamma}}$ with
\[
\mathrm{dom}(\mathcal{A}_{u}^{q,q_{\Gamma}}) = \{ \psi \in \mathrm{W}_0^{-1,q,q_{\Gamma}} \cap \mathrm{W}^{1,2,2} : \mathcal{A}_u\psi \in \mathrm{W}_0^{-1,q,q_{\Gamma}}\},
\]
and $\mathcal{A}_{u}^{q,q_{\Gamma}}= \mathcal{A}_u|_{\mathrm{dom}(\mathcal{A}_{u}^{q,q_{\Gamma}})}$. It will be shown in Lemma \ref{Aiso} below that $\mathrm{dom}(\mathcal{A}_{u}^{q,q_{\Gamma}}) = \mathrm{W}^{1,q,q_{\Gamma}}$. 
 Let $\mathcal{L}_{u}^{q,q_{\Gamma}}$
be the divergence operator in $\mathrm{W}_0^{-1,q,q_{\Gamma}}$ analogously
induced by $\mathfrak{l}_{u}$ and let $\mathcal{L}_{u,+}^{q,q_{\Gamma}}$,
$\mathcal{L}_{u,-}^{q,q_{\Gamma}}$ and $\mathcal{L}_{u,\Gamma}^{q,q_{\Gamma}}$
be the Neumann operators induced by $\mathfrak{l}_{u,+}$, $\mathfrak{l}_{u,-}$
and $\mathfrak{l}_{u,\Gamma}$ on the domains $\op$, $\om$ and $\Gamma$,
respectively. We write $\mathcal{M}_{u}^{q,q_{\Gamma}}$ for the bounded
transmission operator given by 
\begin{equation}
\mathcal{M}_{u}^{q,q_{\Gamma}}(\psi)(\varphi):=\mathfrak{m}_{u}(\psi,\varphi),\qquad\psi\in\mathrm{dom}(\mathcal{L}_{u}^{q,q_{\Gamma}}),\varphi\in\mathrm{W}^{1,q',q'_{\Gamma}},\label{eq:defM}
\end{equation}
so that 
\[
\mathcal{A}_{u}^{q,q_{\Gamma}}=\mathcal{L}_{u}^{q,q_{\Gamma}}+\mathcal{M}_{u}^{q,q_{\Gamma}}.
\]
The external forces and inhomogeneous Neumann boundary conditions
$f,g,h$ in \eqref{eq:plus}–\eqref{eq:Gamma} are realized as a $\mathrm{W}_0^{-1,q,q_{\Gamma}}$-functional
$\mathcal{F}(u)$ with components $\mathcal{F}_{+}(u)\in W_0^{-1,q}(\op)$,
$\mathcal{F}_{-}(u)\in W_0^{-1,q}(\om)$ and $\mathcal{F}_{\Gamma}(u)\in W_0^{-1,q_{\Gamma}}$
given by 
\begin{align*}
\mathcal{F}_{+}(u)(\varphi_{+}) & =\int_{\op}f_{+}(u_{+})\varphi_{+}\,\mathrm{d}x+\int_{\Gamma}g_{+}(u)\tr\varphi_{+}\,\mathrm{d}\mathcal{H}_{d-1}+\int_{\partial\op\setminus\Gamma}h_{+}(u_{+})\mathrm{tr}_{\partial\op\setminus\Gamma}\,\varphi_{+}\,\mathrm{d}\mathcal{H}_{d-1}\\
\mathcal{F}_{-}(u)(\varphi_{-}) & =\int_{\om}f_{-}(u_{-})\varphi_{-}\,\mathrm{d}x+\int_{\Gamma}g_{-}(u)\tr\varphi_{-}\,\mathrm{d}\mathcal{H}_{d-1}+\int_{\partial\om\setminus\Gamma}h_{-}(u_{-})\mathrm{tr}_{\partial\om\setminus\Gamma}\,\varphi_{-}\,\mathrm{d}\mathcal{H}_{d-1}\\
\mathcal{F}_{\Gamma}(u)(\varphi_{\Gamma}) & =\int_{\Gamma}f_{\Gamma}(u)\varphi_{\Gamma}\,\mathrm{d}\mathcal{H}_{d-1}+\int_{\partial\Gamma}h_{\Gamma}(u_{\Gamma})\mathrm{tr}_{\partial\Gamma}\,\varphi_{\Gamma}\,\mathrm{d}\mathcal{H}_{d-2},
\end{align*}
for all $\varphi\in\mathrm{W}^{1,q',q'_{\Gamma}}$. Using trace and
embedding results, it follows that $\mathcal{F}(u)$ is well-defined,
if
\[
f(u)\in\mathrm{L}^{p,p_{\Gamma}},\text{ and }g_{\pm}(u),h_{\pm}(u)\in L^{\rho}(\Gamma),
\]
where 
\[
p>\frac{d}{d+1-\frac{d}{q'}},p_{\Gamma}>\frac{d-1}{d-\frac{d-1}{q_{\Gamma}'}}\text{ and }\rho>\frac{d-1}{d-\frac{d}{q'}}.
\]
If $d=2$, then $p_{\Gamma}>1$ is sufficient. If $d=3$, then $h_{\Gamma}(u_{\Gamma})\in L^{\rho_{\Gamma}}(\partial\Gamma)$
with $\rho_{\Gamma}>1$ is a sufficient condition. With the assumptions
in this section, the system (\ref{eq:plus}), (\ref{eq:minus}) and
(\ref{eq:Gamma}) can be considered as the quasilinear abstract Cauchy
problem 
\begin{equation}
\dot{u}(t)+\mathcal{A}_{u(t)}u(t)=\mathcal{F}(u(t))\in\mathrm{W}_0^{-1,q,q_{\Gamma}},\quad u(0)=u_{0},\label{eq:qlACP}
\end{equation}
for $q,q_{\Gamma}\geq2$.

\subsection{Maximal parabolic regularity and embedding theorems}

\label{maxreg}

In the proof of the main result, we use non-autonomous maximal parabolic
regularity of $\mathcal{A}_{u(\cdot)}^{q,q_{\Gamma}}$ in $\mathrm{W}_0^{-1,q,q_{\Gamma}}$.
To make this more precise, this subsection contains some definitions
and preliminary results. For $T>0$, let in the following $J_{T}=\left(0,T\right)$.
For two Banach spaces $X,Y$ that form an interpolation couple, $(X,Y)_{\theta,p}$
denotes the real interpolation spaces with parameters $\theta\in(0,1)$,
$p\in[1,\infty]$. \begin{definition} Let $1<r<\infty$ , let $X$
be a Banach space and assume that $B$ is a closed operator in $X$
with dense domain $\mathrm{dom}(B)\subset X$, equipped with the graph
norm. We say that $B$ satisfies maximal $L^{r}(J_{T};X)$-regularity
if for all $u^{0}\in(\mathrm{dom}(B),X)_{1-\frac{1}{r},r}$ and $f\in L^{r}(0,T;X)$
there is a unique solution 
\[
u\in L^{r}(J_{T};\mathrm{dom}(B))\cap W^{1,r}(J_{T};X)
\]
of the abstract Cauchy problem 
\[
\left\{ \begin{array}{rcll}
\dot{u}+Bu & = & f,\\
u(0) & = & u^{0},
\end{array}\right.
\]
posed in $X$, satisfying 
\[
\Vert\dot{u}\Vert_{L^{r}(J_{T};X)}+\Vert Bu\Vert_{L^{r}(J_{T};X)}\leq C(\Vert u^{0}\Vert_{(\mathrm{dom}(B),X)_{1-\frac{1}{r},r}}+\Vert f\Vert_{L^{r}(J_{T};X)})
\]
with a constant $C>0$ independent of $u^{0}$ and $f$ (see e.g.
\citep[Ch. III.1]{Amann95}). \end{definition} Note that the notion
of maximal $L^{r}(J_{T};X)$-regularity is independent of $1<r<\infty$
and $T>0$, cf. \citep{Dore2000}. In the following, for $q,q_{\Gamma}\geq2$,
$1<r<\infty$ and given $u\in\mathrm{C}^{0,0}$, we consider maximal
regularity of $\mathcal{A}_{u}^{q,q_{\Gamma}}$. Thus, 
\[
\mathrm{MR}_{q,q_{\Gamma}}^{r}:=L^{r}(J_{T};\mathrm{dom}(\mathcal{A}_{u}^{q,q_{\Gamma}}))\cap W^{1,r}(J_{T};\mathrm{W}_0^{-1,q,q_{\Gamma}})
\]
is the corresponding solution space and 
\[
\mathrm{X}_{q,q_{\Gamma}}^{r}:=(\mathrm{dom}(\mathcal{A}_{u}^{q,q_{\Gamma}}),\mathrm{W}_0^{-1,q,q_{\Gamma}})_{1-\frac{1}{r},r}
\]
is the corresponding time trace space.

In Lemma \ref{Aiso} below, it is shown that there are $q>d$, $q_{\Gamma}>d-1$,
such that $\mathrm{dom}(\mathcal{A}_{u}^{q,q_{\Gamma}})=\mathrm{W}^{1,q,q_{\Gamma}}$.
The following lemma summarizes useful embeddings for the corresponding
function spaces.

\begin{lemma} \label{lem:Cbb} If $\mathrm{dom}(\mathcal{A}_{u}^{q,q_{\Gamma}})=\mathrm{W}^{1,q,q_{\Gamma}}$,
then 
\begin{enumerate}
\item for $\alpha\leq1-\frac{d}{q}$ and $\alpha_{\Gamma}\leq1-\frac{d-1}{q_{\Gamma}}$,
\begin{equation}
\mathrm{dom}(\mathcal{A}_{u}^{q,q_{\Gamma}})\hookrightarrow\mathrm{C}^{\alpha,\alpha_{\Gamma}},\label{eq:alpha}
\end{equation}
\item for any $1<r<\infty$, 
\begin{equation}
\mathrm{MR}_{q,q_{\Gamma}}^{r}\hookrightarrow C^{0}(J_{T};\mathrm{X}_{q,q_{\Gamma}}^{r}).\label{eq:timetrace}
\end{equation}
If $q>d$, $q_{\Gamma}>d-1$, and $r>\max(\frac{2q}{q-d},\frac{2q_{\Gamma}}{q_{\Gamma}-d+1})$,
then 
\begin{equation}
\mathrm{X}_{q,q_{\Gamma}}^{r}\hookrightarrow\mathrm{C}^{\beta,\beta_{\Gamma}},\label{eq:beta}
\end{equation}
where $0<\beta\leq1-\frac{d}{q}-\frac{2}{r}$ and $0<\beta_{\Gamma}\leq1-\frac{d-1}{q_{\Gamma}}-\frac{2}{r}$. 
\item for $q>d$, $q_{\Gamma}>d-1$, let $0<\delta<\min(\frac{q-d}{2q},\frac{q_{\Gamma}-d+1}{2q_{\Gamma}})$
and $r>\max(\frac{2q}{q-2\delta q-d},\frac{2q_{\Gamma}}{q_{\Gamma}-2\delta q_{\Gamma}-d+1})$,
then 
\begin{equation}
\mathrm{MR}_{q,q_{\Gamma}}^{r}\hookrightarrow C^{\delta}(J_{T};\mathrm{C}^{\gamma,\gamma_{\Gamma}})\label{eq:gamma}
\end{equation}
with $0<\gamma\leq1-\frac{d}{q}-\frac{2}{r}-2\delta$ and $0<\gamma_{\Gamma}\leq1-\frac{d-1}{q_{\Gamma}}-\frac{2}{r}-2\delta$.
In particular, the embedding 
\begin{equation}
\mathrm{MR}_{q,q_{\Gamma}}^{r}\hookrightarrow C^{0}(J_{T};\mathrm{C}^{0,0})\label{eq:compact}
\end{equation}
is compact. 
\end{enumerate}
\end{lemma}

\begin{proof} Note that $\op,\om$ and $\Gamma$ are sufficiently
regular for embedding and interpolation results to work ``as usual''.
The first embedding \eqref{eq:alpha} is standard, cf. e.g.\,\citep[2.8.1(c)]{Triebel95}.
For embedding \eqref{eq:timetrace}, cf.\,\citep[Section III.4.10]{Amann95}.
Embedding \eqref{eq:beta} follows by definition of $\mathrm{X}_{q,q_{\Gamma}}^{r}$,
combining e.g.\, the interpolation result \citep[p. 186, (14)]{Triebel95}
and the embedding \citep[2.8.1]{Triebel95}. From \citep[Lemma 3.4(b)]{DRtE15Hoel},
it follows that 
\[
\mathrm{MR}_{q,q_{\Gamma}}^{r}\hookrightarrow C^{\delta}(J_{T};(\mathrm{W}_0^{-1,q,q_{\Gamma}},\mathrm{W}^{1,q,q_{\Gamma}})_{\theta,1})
\]
with $0<\theta\leq1-\frac{1}{r}-\delta$. Embedding \eqref{eq:gamma}
then follows again by combining \citep[p. 186, (14)]{Triebel95} and
\citep[2.8.1]{Triebel95}. \end{proof} For uniqueness, an assumption
on the dependence of $\mathcal{F}$ of $u$ is needed. By embedding,
the assumption is satisfied, e.g. if the dependence of $\mathcal{F}$
of $u$ in an $L^{p}$-norm is locally Lipschitz (for example, it
is straightforward to see that semilinear terms of reaction-diffusion
type with arbitrary powers are included \citep{Disser19}). \begin{assumption}\label{Lipschitz}
Given $1<r<\infty$ and $q,q_{\Gamma}>2$, the function $\mathcal{F}\colon\mathrm{X}_{q,q_{\Gamma}}^{r}\to\mathrm{W}_0^{-1,q,q_{\Gamma}}$
is boundedly Lipschitz in the sense that for all $\tilde{L}>0$, there
exists a constant $C_{\tilde{L}}>0$ such that for all $u_{1},u_{2}\in\mathrm{X}_{q,q_{\Gamma}}^{r}$
with $\Vert u_{1}\Vert_{\mathrm{X}_{q,q_{\Gamma}}^{r}},\Vert u_{2}\Vert_{\mathrm{X}_{q,q_{\Gamma}}^{r}}\leq\tilde{L}$,
\[
\Vert\mathcal{F}(u_{1})-\mathcal{F}(u_{2})\Vert_{\mathrm{W}_0^{-1,q,q_{\Gamma}}}\leq C_{\tilde{L}}\Vert u_{1}-u_{2}\Vert_{\mathrm{X}_{q,q_{\Gamma}}^{r}}.
\]
\end{assumption}

\section{Global existence and uniqueness \label{sec:The-Sesquilinear-form}}

The main result of the paper is global existence and uniqueness of
solutions of \eqref{eq:qlACP}. For local well-posedness, it is sufficient
that $\mathcal{F}$ satisfies Assumption \ref{Lipschitz}. For global
existence, it is required that $\mathcal{F}$ preserves the $L^{\infty}$-bounds
for \eqref{eq:qlACP}. This requirement on $\mathcal{F}$ is defined more precisely in Step \eqref{proof:mp} in the proof of Theorem \ref{thm:local}. Examples of $\mathcal{F}$ that satisfy these
assumptions are given in Corollary \ref{rhsmp} and it is shown that
this includes coupled volume-surface reaction-diffusion in \citep{Disser19}
(the case $\mathcal{F}=0$ corresponds to the original gradient structure).
\begin{thm} \label{thm:local}There exist $q>d$, $q_{\Gamma}>d-1$
such that for all $r>\max(\frac{2q}{q-d},\frac{2q_{\Gamma}}{q_{\Gamma}-d+1})$,
$u^{0}\in\mathrm{X}_{q,q_{\Gamma}}^{r}$, $T>0$ and $\mathcal{F}$
satisfying Assumption \ref{Lipschitz} and preserving $L^{\infty}$-bounds,
there is a unique global solution 
\[
u\in W^{1,r}(J_{T};\mathrm{W}_0^{-1,q,q_{\Gamma}})\cap L^{r}(J_{T};\mathrm{W}^{1,q,q_{\Gamma}})
\]
of (\ref{eq:qlACP}). 
In particular, the solution is Hölder continuous in time and space,
\[
u\in C^{\delta}(J_{T};\mathrm{C}^{\gamma,\gamma_{\Gamma}}),
\]
with $\delta,\gamma,\gamma_{\Gamma}$ as in Lemma \ref{lem:Cbb} and
(\ref{eq:qlACP}) is well-posed. \end{thm}

The proof of Theorem \ref{thm:local} is divided into four steps: 
\begin{enumerate}
\item \label{proof:bc} provisional reduction to bounded coefficients, 
\item \label{proof:nA} preliminary results on the linearized non-autonomous
problem, 
\item \label{proof:mp} a priori $L^{\infty}$-bounds, 
\item \label{proof:thm} Schaefer argument and proof of the theorem. 
\end{enumerate}

\subsection*{\eqref{proof:bc} Provisional reduction to bounded coefficients}

By Lemma \ref{lem:Cbb}, $u^{0}\in\mathrm{C}^{\beta,\beta_{\Gamma}}\subset\mathrm{C}^{0,0}$.
Let $-\infty<l_{0}\leq L_{0}<+\infty$ be such that $u^{0}\in\mathrm{C}_{l_{0}}^{L_{0}}$.
Define 
\[
[f]_{l}^{L}(x):=\begin{cases}
\begin{array}{cc}
L, & f(x)\geq L,\\
l, & f(x)\leq l,\\
f(x), & \text{otherwise.}
\end{array}\end{cases}
\]
and let $L:=L_{0}+1$, $l=l_{0}/2$. Instead of the coefficient functions
$k$ and $m$, consider 
\begin{equation}
k_{l}^{L}(\cdot,u(\cdot))=k(\cdot,[u]_{l}^{L}(\cdot))\quad\text{and}\quad m_{l}^{L}(\cdot,u(\cdot))=m(\cdot,[u]_{l}^{L}(\cdot))\label{eq:KL}
\end{equation}
in the following. In Step \eqref{proof:thm} below, it is shown that
$k_{l}^{L}=k$ and $m_{l}^{L}=m$ along the orbits of $u^{0}$. Since
the solution is unique and regular and the dependence of $k,m$ of
the solution is Lipschitz, this is sufficient to prove the theorem.
Clearly, if $k,m$ satisfy Assumption \ref{Akm}, then also $k_{l}^{L}$,
$m_{l}^{L}$ satisfy Assumption \ref{Akm} and the bounds in \ref{Akm}\eqref{akm1}
and \ref{Akm}\eqref{akm2} hold uniformly in $u\in\mathrm{C}^{0,0}$
for $k_{l}^{L}$, $m_{l}^{L}$.

\subsection*{\eqref{proof:nA} Preliminary results on the linearized non-autonomous
problem}

In this step of the proof and in Step \eqref{proof:mp}, using Step
\eqref{proof:bc}, assume additionally that all coefficient functions
are such that the bounds in \ref{Akm}\eqref{akm1} and \ref{Akm}\eqref{akm2}
hold uniformly in $u\in\mathrm{C}^{0,0}$.

\begin{lemma}\label{Aiso} There exist $q>d$ and $q_{\Gamma}>d-1$
such that for any $u\in C^{0}(J_{T};\mathrm{C}^{0,0})$, for all $t\in\overline{J_{T}}$,
for any $\lambda>0$, the operator $\mathcal{A}_{u(t)}^{q,q_{\Gamma}}+\lambda$
is an isomorphism 
\begin{equation}
\mathcal{A}_{u(t)}^{q,q_{\Gamma}}+\lambda\colon\mathrm{W}^{1,q,q_{\Gamma}}\to\mathrm{W}_0^{-1,q,q_{\Gamma}}.
\end{equation}
\end{lemma} \begin{proof} First note that $\mathcal{A}_{u(t)}^{q,q_{\Gamma}}\colon\mathrm{dom}(\mathcal{A}_{u(t)}^{q,q_{\Gamma}})\to\mathrm{W}_0^{-1,q,q_{\Gamma}}$
is well-defined for all $t\in\overline{J_{T}}$. In the case $d=2$,
by the Lax-Milgram theorem, the claim holds for $q=q_{\Gamma}=2$.
By Sneiberg's theorem \citep{snei}, the isomorphism property extrapolates
to a neighbourhood of $\mathrm{W}^{1,2,2}$ in the complex interpolation
scale $[\mathrm{W}^{1,p,p_{\Gamma}},\mathrm{W}^{1,p',p'_{\Gamma}}]_{1/p}=\mathrm{W}^{1,2,2_{\Gamma}},\,1<p,p_{\Gamma}<\infty$,
see \citep{Groeger89}. \\
If $d=3$, then Assumption \ref{Akm}\eqref{akm4} holds. If $\kappa_{\pm}\equiv1$,
then $k_{\pm}=\varkappa_{\pm}$ is independent of $u$ and then by \citep[Theorem~1.1]{ers},
\citep[Lemma~6.5]{HDR09}, there is a $q>3$ such that the isomorphism
property $\mathcal{L}_{u(t),\pm}+\lambda\colon W^{1,q}(\Omega_{\pm})\to W_0^{-1,q}(\Omega_{\pm})$
holds true. Using the same extrapolation argument as in the case $d=2$,
there exists a $q_{\Gamma}>2$ such that $\mathcal{L}_{u(t),\Gamma}+\lambda\colon W^{1,q_{\Gamma}}(\Gamma)\to W_0^{-1,q_{\Gamma}}(\Gamma)$
is an isomorphism. In \citep[Theorem~6.3]{DisKaiReh15OSR} it was
shown that the domains of $\mathcal{L}_{u(t),\pm},\mathcal{L}_{u(t),\Gamma}$
remain unchanged by a scalar multiplicative perturbation $\kappa_{\pm}\in C^{0}(\Omega_{\pm})$
that is positively bounded from below. This proves the result for
the operators $\mathcal{L}_{u(t)}$, $t\in\overline{J_{T}}$. By relative
boundedness of $\mathcal{M}_{u(t)}$, \citep[Lemma~3.4]{DisDIMX15},
the domains of $\mathcal{L}_{u(t)}+\lambda$ and $\mathcal{A}_{u(t)}+\lambda$
coincide. This proves the claim. \end{proof} 

\begin{lemma}\label{Amr}Let $2\leq q,q_{\Gamma}<\infty$, $1<r<\infty$
and let $u\in C^{0}(J_{T};\mathrm{C}^{0,0})$. Then for all $t\in\overline{J_{T}}$,
$\mathcal{A}_{u(t)}^{q,q_{\Gamma}}$ has maximal $L^{r}(J_{T};\mathrm{W}_0^{-1,q,q_{\Gamma}})$-regularity.
\end{lemma} \begin{proof} The result was shown in \citep{DisDIMX15}
if $\Gamma$ is flat. It remains to check the maximal regularity of
the Neumann operator $\mathcal{L}_{u(t),\Gamma}$ on $C^{1}$-boundaries.
This follows from maximal regularity for flat domains \citep{Groeger89},
using the usual localization methods, i.e. exploiting that the property
of maximal regularity is preserved under perturbations that occur
when locally flattening the domain and straightening the boundary
with respect to a sufficiently fine covering and a corresponding partition
of unity, see \citep{DHP03} for the general strategy and \citep{DenkPrussZacher08}
for this argument in a similar context. \end{proof} \begin{lemma}
\label{AnAmr} Let $w\in C^{0}(J_{T};\mathrm{C}^{0,0})$, $q,q_{\Gamma}$
be as in Lemma \ref{Aiso}. Then for every $r$ and $u^{0}\in\mathrm{X}_{q,q_{\Gamma}}^{r}$
as in Theorem \ref{thm:local}, for all $f\in L^{r}(J_{T};\mathrm{W}_0^{-1,q,q_{\Gamma}})$,
there exists a unique global solution $v\in\mathrm{MR}_{q,q_{\Gamma}}^{r}$
of 
\begin{align}
\dot{v}(t)+\mathcal{A}_{w(t)}^{q,q_{\Gamma}}v(t) & =f(t),\quad\text{in }\mathrm{W}_0^{-1,q,q_{\Gamma}},\label{nonA}\\
v(0) & =u^{0}.\nonumber 
\end{align}
The solution operator 
\begin{equation}
(\partial_{t}+\mathcal{A}_{w(\cdot)}^{q,q_{\Gamma}})^{-1}\colon f\in L^{r}(J_{T};\mathrm{W}_0^{-1,q,q_{\Gamma}})\mapsto v\in\mathrm{MR}_{q,q_{\Gamma}}^{r}
\end{equation}
is bounded. \end{lemma} \begin{proof} For two Banach spaces $X,Y$,
let $\mathcal{B}(X,Y)$ denote the space of bounded linear operators
$B\colon X\to Y$. By continuity of $w$, $k_{l}^{L}$ and $m_{l}^{L}$
and by Lemma \ref{Aiso}, the map $J_{T}\ni t\mapsto\mathcal{A}_{w}^{q,q_{\Gamma}}(t)\in\mathcal{B}(\mathrm{W}^{1,q,q_{\Gamma}},\mathrm{W}_0^{-1,q,q_{\Gamma}})$
is uniformly continuous. By Lemma \ref{AnAmr}, for all $t\in\overline{J_{T}}$,
$\mathcal{A}_{w}^{q,q_{\Gamma}}(t)$ has maximal $L^{r}(J_{T};\mathrm{W}_0^{-1,q,q_{\Gamma}})$-regularity,
so existence and boundedness of the solution operator follow from
\citep[Theorem~2.5]{PruSch01nonAut}. \\
 \end{proof} 

\subsection*{\eqref{proof:mp} $L^{\infty}$-bounds on $u$}

If $\mathcal{F}=0,$ uniform $\mathrm{L}^{\infty,\infty}$-bounds
on $u$ can be proved by a bulk-interface comparison principle. With
respect to results for the bulk problem, the point is to show that
the nonlinear bulk-interface interaction terms derived from a generalized
gradient structure preserve this property. \begin{lemma} \label{mp}
Let $r,q,q_{\Gamma}$ as in Theorem \ref{thm:local}, $\mathcal{F}=0$
and $u^{0}\in\mathrm{X}_{q,q_{\Gamma}}^{r}$ with $u^{0}\in\mathrm{C}_{l}^{L}$
for some $-\infty<l\leq L<+\infty$. Assume that $u\in\mathrm{MR}_{q,q_{\Gamma}}^{r}$
is a solution of \eqref{eq:qlACP}. 
 Then for all $t\in\overline{J_{T}}$, $u(t)\in\mathrm{C}_{l}^{L}.$
\end{lemma} 
\begin{proof} Define $\zeta_{l}(t)=[(u(t)-l)^{-}]$ and $\zeta^{L}(t)=[(L-u(t))^{-}]$,
where 
\[
[f^{-}](x):=\begin{cases}
\begin{array}{cc}
0, & f(x)\geq0,\\
-f(x), & f(x)<0.
\end{array}\end{cases}
\]
Since $[\cdot^{-}]$ is Lipschitz and $r,q,q_{\Gamma}\geq2$, we have
$\zeta_{l},\zeta^{L}\in L^{r}(J_{T};\mathrm{W}^{1,q,q_{\Gamma}})\hookrightarrow L^{r'}(J_{T};\mathrm{W}^{1,q',q'_{\Gamma}})$
with
\[
\nabla\zeta^{L}(t,x)=\begin{cases}
\begin{array}{cc}
0, & u(t,x)\leq L,\\
\nabla u(t,x), & u(t,x)>L,
\end{array}\end{cases}
\]
and $\zeta_{l}(0)=\zeta^{L}(0)\equiv0$. Thus, for all $s\in\overline{J_{T}}$,
testing \eqref{eq:qlACP} with $\zeta^{L}$ in space and time gives
$\zeta^{L}\equiv0$ as 
\[
\int_{0}^{s}\dot{u}(t)(\zeta^{L}(t))\,\mathrm{d}t=\frac{1}{2}\Vert\zeta^{L}(s)\Vert_{\mathrm{L}^{2,2}}^{2}\geq0
\]
and 
\begin{equation}
\int_{0}^{s}\mathcal{A}_{u(t)}^{q,q_{\Gamma}}u(t)(\zeta^{L}(t))\,\mathrm{d}t=\int_{0}^{s}\mathfrak{l}_{u(t)}(u(t),\zeta^{L}(t))+\mathfrak{m}_{u(t)}(u(t),\zeta^{L}(t))\,\mathrm{d}t\geq0.\label{eq:maxprinc}
\end{equation}
To show the estimate from below in (\ref{eq:maxprinc}), note that
\[
\int_{0}^{s}\mathfrak{l}_{u(t)}(u(t),\zeta^{L}(t))\,\mathrm{d}t=\int_{0}^{s}\mathfrak{l}_{u(t)}(\zeta^{L}(t),\zeta^{L}(t))\,\mathrm{d}t\geq0
\]
as $k$ is bounded below by $\underline{k}$ and that 
\begin{align}
\int_{0}^{s}\mathfrak{m}_{u(t)}(u(t),\zeta^{L}(t))\,\mathrm{d}t & =\int_{0}^{s}\int_{\Gamma}m_{+}(u)(\tp-\tg)(\zeta_{+}^{L}-\zeta_{\Gamma}^{L})(t)\label{marray}\\
 & +m_{-}(u)(\tm-\tg)(\zeta_{-}^{L}-\zeta_{\Gamma}^{L})(t)\nonumber \\
 & +m_{\Gamma}(u)(\tp-\tm)(\zeta_{+}^{L}-\zeta_{-}^{L})(t)\,\mathrm{d}\mathcal{H}_{d-1}\,\mathrm{d}t,\nonumber 
\end{align}
where $m$ is bounded below by $\underline{m}$ and where 
\begin{align*}
\int_{\Gamma}(\tp-\tg)(\zeta_{+}^{L}-\zeta_{\Gamma}^{L})(t)\,\mathrm{d}\mathcal{H}_{d-1} & =\int_{\{x\in\Gamma\colon\tp(x)>L>\tg(x)\}}(\tp-\tg)(\tp-L)(t)\,\mathrm{d}\mathcal{H}_{d-1}\\
 & +\int_{\{x\in\Gamma\colon\tp(x)<L<\tg(x)\}}(\tp-\tg)(L-\tg)(t)\,\mathrm{d}\mathcal{H}_{d-1}\\
 & +\int_{\{x\in\Gamma\colon\tp(x),\tg(x)>L\}}(\tp-\tg)(\tp-\tg)(t)\,\mathrm{d}\mathcal{H}_{d-1}\geq0,\\
\end{align*}
and non-negativity of the remaining terms on the right-hand-side of
\eqref{marray} follows analogously. \\
 The proof of the lower bound, i.e. $\zeta^{l}\equiv0$ follows analogously
by testing \eqref{eq:qlACP} with $\zeta^{l}$.

\end{proof}

If $\mathcal{F}\neq0$, the proof of Theorem \ref{thm:local} still
requires that $\mathcal{F}$ preserves $L^{\infty}$-bounds in a suitable
sense. More precisely, it is straightforward to see that the proof of Lemma \ref{mp} still works if $\mathcal{F}$ is such that for given $u_0 \in \mathrm{C}_l^L$, there are constants $l_F < L_F \in \mathbb{R}$ such that for all $u \in \mathrm{MR}^r_{q,q_\Gamma}$ such that $u(0)=u_0$, for all $s \in J_T$, 
$$ \int_{0}^{s}\mathcal{F}(u(t))(\zeta^{L_{f}}(t))\,\mathrm{d}t\leq0 \quad \text{and} \quad
\int_{0}^{s}\mathcal{F}(u(t))(\zeta_{l_{f}}(t))\,\mathrm{d}t\geq0.$$
It is shown in \citep{Disser19} how chemical reaction rates
of type \eqref{eq:reactionRates} fit into the framework of Theorem
\ref{thm:local}. Another general example is given by terms of Allen-Cahn-type,
treated in the following corollary. A concrete example is 
\begin{equation}
f_{+}(u_{+})=-u_{+}^{3},g_{+}(u_{+})=(1-u_{+})^{3},\quad f_{\Gamma}(u_{\Gamma})=-u_{\Gamma}^{5}.\label{eq:AllenCahnEx}
\end{equation}
\begin{coro}\label{rhsmp} Let $\mathcal{F}$ satisfy Assumption
\ref{Lipschitz} and let all the components $\varphi$ of $\mathcal{F}$,
e.g. $\varphi=f_{+},g_{-},\dots$ in \eqref{eq:plus}–\eqref{eq:Gamma}
be independent of $x\in\op,\om$, $y\in\Gamma$, respectively. Assume
that $\varphi$ are continuously differentible in $u$ and that $g_{\pm}$
depend only on $u_{\pm}$, respectively, whereas $f_{\Gamma}$ depends
only on $u_{\Gamma}$. Assume that all $\varphi$ satisfy the dissipativity
condition 
\begin{equation}
\liminf_{\vert v\vert\to\infty}-\varphi'(v)>0.\label{dissCond}
\end{equation}
Then, under the assumptions of Lemma \ref{mp}, given a solution $u\in\mathrm{MR}_{q,q_{\Gamma}}^{r}$
of \eqref{eq:qlACP}, there are constants $-\infty<l_{f}\leq L_{f}<+\infty$,
such that for all $t\in\overline{J_{T}}$, $u(t)\in\mathrm{C}_{l_{f}}^{L_{f}}$.
\end{coro} \begin{proof} Condition \eqref{dissCond} guarantees
that for every component $\varphi$, there exist constants $-\infty<l_{\varphi}\leq L_{\varphi}<+\infty$
such that $\varphi(v)>0$ for all $v<l_{\varphi}$ and $\varphi(v)<0$
for all $v>L_{\varphi}$. Let $l_{f}:=\min_{\varphi}(l_{\varphi})$
and $L^{f}:=\max_{\varphi}(L_{\varphi})$. In the choice of test functions
$\zeta_{l},\zeta^{L}$ in the proof of Lemma \ref{mp}, replace $l,L$
by $l_{f},L_{f}$. It is then straightforward to check that for all
$s\in\overline{J_{T}}$, $\int_{0}^{s}\mathcal{F}(u(t))(\zeta^{L_{f}}(t))\,\mathrm{d}t\leq0$
and that $\int_{0}^{s}\mathcal{F}(u(t))(\zeta_{l_{f}}(t))\,\mathrm{d}t\geq0$.
Combined with the calculations in the proof of Lemma \ref{mp}, this
proves the claim. \end{proof} 

\subsection*{\eqref{proof:thm} Schaefer argument and proof of Theorem \ref{thm:local}}

Let $q,q_{\Gamma}$ be given by Lemma \ref{Aiso} and let $r$ and
$u^{0}\in\mathrm{X}_{q,q_{\Gamma}}^{r}$ be given as in Theorem \ref{thm:local}.
By embedding \eqref{eq:beta}, $u^{0}\in\mathrm{C}_{l}^{L}$ for some
$-\infty<l\leq L<+\infty$. In the following, let 
\begin{equation}
C_{u^{0}}^{0}(J_{T};\mathrm{C}^{0,0}):=\{u\in C^{0}(J_{T};\mathrm{C}^{0,0})\colon u(0)=u^{0}\}.
\end{equation}
Define 
\[
\mathcal{T}\colon C_{u^{0}}^{0}(J_{T};\mathrm{C}^{0,0})\to C_{u^{0}}^{0}(J_{T};\mathrm{C}^{0,0})
\]
by $\mathcal{T}w=v\in\mathrm{MR}_{q,q_{\Gamma}}^{r}$ the solution
of \eqref{nonA} with $v(0)=u^{0}$ given by Lemma \ref{AnAmr}. By
embedding \eqref{eq:compact}, Lemma \ref{AnAmr} and Assumption \ref{Lipschitz},
\[
\mathcal{T}w=\mathrm{Id}_{\mathrm{MR}_{q,q_{\Gamma}}^{r}\to C_{u^{0}}^{0}(J_{T};\mathrm{C}^{0,0})}(\partial_{t}+\mathcal{A}_{w(\cdot)}^{q,q_{\Gamma}})^{-1}(\mathcal{F}(w))
\]
is well-defined and compact. A fixed point of $\mathcal{T}$ would
solve \eqref{eq:qlACP}. To obtain existence of a fixed point by Schaefer's
Theorem \citep[Theorem~9.2.4]{Evans2010PDE}, it suffices to show
that
\begin{enumerate}
\item $\mathcal{T}$ is continuous, and that 
\item \label{enu:the-Schaefer-set}the Schaefer set 
\[
\mathrm{S}:=\{u\in C_{u^{0}}^{0}(J_{T};\mathrm{C}^{0,0}):u=\lambda\mathcal{T}(u)\text{ for some }0\leq\lambda\leq1\}
\]
is bounded. 
\end{enumerate}
To show continuity of $\mathcal{T}$, we show continuity of the map
\[
C_{u^{0}}^{0}(J_{T};\mathrm{C}^{0,0})\ni w\mapsto(\partial_{t}+\mathcal{A}_{w(\cdot)}^{q,q_{\Gamma}})^{-1}\in\mathcal{B}(L^{r}(J_{T};\mathrm{W}_0^{-1,q,q_{\Gamma}});\mathrm{MR}_{q,q_{\Gamma}}^{r})=:\mathcal{B}_{q,q_{\Gamma}}^{r}.
\]
This follows from bounded Lipschitzianity of the map $w\mapsto\mathcal{A}_{w(\cdot)}^{q,q_{\Gamma}}$
and continuity of the inversion. A detailed proof is the following:
for given $w\in C_{u^{0}}^{0}(J_{T};\mathrm{C}^{0,0})$, there are
constants $-\infty<l_{0}<L_{0}<+\infty$ (possibly dependent on $T$),
such that $w\in C_{u^{0}}^{0}(J_{T};\mathrm{C}_{l_{0}}^{L_{0}}).$
Define 
\[
C_{w}:=\Vert(\partial_{t}+\mathcal{A}_{w(\cdot)}^{q,q_{\Gamma}})^{-1}\Vert_{\mathcal{B}_{q,q_{\Gamma}}^{r}}
\]
(Lemma \ref{AnAmr}) and let $C_{\delta}$ be a Lipschitz constant
for $k$ and $m$ on $(\mathbb{R}_{l_{0}-\delta}^{L_{0}+\delta})^{3}$
with $\delta>0$ ($\delta$ possibly small such that Assumption \ref{Akm}\eqref{akm3}
applies). Let $w_{n}\to w$ in $C_{u^{0}}^{0}(J_{T};\mathrm{C}^{0,0})$
and let $n$ be so large that $w_{n}\in C_{u^{0}}^{0}(J_{T};\mathrm{C}_{l_{0}-\delta}^{L_{0}+\delta})$
and $\Vert w_{n}-w\Vert_{C_{u^{0}}^{0}(J_{T};\mathrm{C}^{0,0})}\leq\frac{1}{2C_{w}C_{\delta}T^{1/r}}$.
Then by Assumption \ref{Akm}\eqref{akm3} and the definition of $\mathcal{A}_{w}^{q,q_{\Gamma}}$,
\begin{equation}
\Vert\mathcal{A}_{w(\cdot)}^{q,q_{\Gamma}}-\mathcal{A}_{w_{n}(\cdot)}^{q,q_{\Gamma}}\Vert_{L^{r}(J_{T};\mathcal{B}(\mathrm{W}^{1,q,q_{\Gamma}},\mathrm{W}_0^{-1,q,q_{\Gamma}})}\leq C_{\delta}T^{1/r}\Vert w-w_{n}\Vert_{C_{u^{0}}^{0}(J_{T};\mathrm{C}^{0,0})},\label{eq:LS}
\end{equation}
and thus
\begin{align*}
 & \Vert(\partial_{t}+\mathcal{A}_{w(\cdot)}^{q,q_{\Gamma}})^{-1}-(\partial_{t}+\mathcal{A}_{w_{n}(\cdot)}^{q,q_{\Gamma}})^{-1}\Vert_{\mathcal{B}_{q,q_{\Gamma}}^{r}}\\
 & =\Vert(\partial_{t}+\mathcal{A}_{w(\cdot)}^{q,q_{\Gamma}})^{-1}(\partial_{t}+\mathcal{A}_{w_{n}(\cdot)}^{q,q_{\Gamma}}-\partial_{t}+\mathcal{A}_{w(\cdot)}^{q,q_{\Gamma}})(\partial_{t}+\mathcal{A}_{w_{n}(\cdot)}^{q,q_{\Gamma}})^{-1}\Vert_{\mathcal{B}_{q,q_{\Gamma}}^{r}}\\
 & \leq\Vert(\partial_{t}+\mathcal{A}_{w(\cdot)}^{q,q_{\Gamma}})^{-1}(\mathcal{A}_{w_{n}(\cdot)}^{q,q_{\Gamma}}-\mathcal{A}_{w(\cdot)}^{q,q_{\Gamma}})\Vert_{\mathcal{B}(\mathrm{MR}_{q,q_{\Gamma}}^{r};\mathrm{MR}_{q,q_{\Gamma}}^{r})}\cdot\\
 & \qquad\cdot\Vert(\partial_{t}+\mathcal{A}_{w(\cdot)}^{q,q_{\Gamma}})^{-1}+(\partial_{t}+\mathcal{A}_{w_{n}(\cdot)}^{q,q_{\Gamma}})^{-1}-(\partial_{t}+\mathcal{A}_{w(\cdot)}^{q,q_{\Gamma}})^{-1}\Vert_{\mathcal{B}_{q,q_{\Gamma}}^{r}}\\
 & \leq C_{w}^{2}C_{\delta}T^{1/r}\Vert w_{n}-w\Vert_{C_{u^{0}}^{0}(J_{T};\mathrm{C}^{0,0})}+\frac{1}{2}\Vert(\partial_{t}+\mathcal{A}_{w(\cdot)}^{q,q_{\Gamma}})^{-1}-(\partial_{t}+\mathcal{A}_{w_{n}(\cdot)}^{q,q_{\Gamma}})^{-1}\Vert_{\mathcal{B}_{q,q_{\Gamma}}^{r}},
\end{align*}
so $\mathcal{T}$ is continuous. 

To show \eqref{enu:the-Schaefer-set}, note that if $u_{\lambda}=\lambda\mathcal{T}(u_{\lambda})$
for some $0<\lambda\leq1$, then by definition of $\mathcal{T}$,
$u_{\lambda}\in\mathrm{MR}_{q,q_{\Gamma}}^{r}$ and $u_{\lambda}$
satisfies \eqref{eq:qlACP} with initial value $u_{\lambda}(0)=\lambda u^{0}$
and right-hand-side $\lambda\mathcal{F}(u_{\lambda})$. Thus, if $\mathcal{F}\equiv0$
or if $\mathcal{F}$ is as in Corollary \ref{rhsmp}, then $\mathrm{S}$
is bounded.

In addition, the $L^{\infty}$-bounds are such that $k_{l}^{L}=k$
and $m_{l}^{L}=m$ along orbits of $u^{0}$, justifying step \eqref{proof:bc}
a posteriori with possible adjustments to the choice of $l$ and $L$
by Corollary \ref{rhsmp}. This concludes the proof of existence in
Theorem \ref{thm:local}.

Local well-posedness and hence also uniqueness follow from the Lipschitz
dependence \eqref{eq:LS} that provides a contraction for small $T>0$,
see \citep[Theorem~3.1]{PruessBari2002} for the abstract result in
the theory of maximal parabolic regularity for quasilinear abstract
Cauchy problems and see \citep{DisDIMX15} for a proof in a very similar
setting. Global stability in $\mathrm{L}^{2,2}$ is shown in the next
section. 

\section{Exponential decay to equilibrium and stability}

\label{stability} We prove that under Assumption \ref{Akm}\eqref{akm2},
the interaction on and across the interface $\Gamma$ is sufficiently
strong to force the system into the uniform equilibrium given by 
\[
u^{\infty}=\frac{1}{V}\left(\int_{\op}u_{+}^{0}(x)\,\mathrm{d}x+\int_{\om}u_{-}^{0}(x)\,\mathrm{d}x+\int_{\Gamma}u_{\Gamma}^{0}(y)\,\mathrm{d}\mathcal{H}_{d-1}\right)
\]
associated to $u^{0}$, where 
\[
V=|\op|+|\om|+|\Gamma|_{\mathcal{H}_{d-1}}.
\]
The quasilinear gradient structure combined with the $L^{\infty}$-bounds
provide an exponential rate in $\mathrm{L}^{2,2}$. Here, by a slight
abuse of notation, $u^{\infty}$ also denotes the constant vector
function $u^{\infty}=u^{\infty}(1,1,1)^{T}\in\mathrm{C}^{0,0}$.

\begin{thm} \label{thm:convToEq}Under the assumptions of Theorem
\ref{thm:local} with $\mathcal{F}\equiv0$, given $u^{0}\in\mathrm{X}_{q,q_{\Gamma}}^{r}$,
the solution $u$ converges to $u^{\infty}$ at an exponential rate,
in the sense that there is a $\delta>0$ depending only on $u^{0},k,m,\Omega$
and $\Gamma$, such that for all $s\geq0$, 
\begin{equation}
\Vert u(s)-u^{\infty}\Vert_{\mathrm{L}^{2,2}}\leq e^{-\delta s}\Vert u^{0}-u^{\infty}\Vert_{\mathrm{L}^{2,2}}.\label{eq:exp}
\end{equation}
\end{thm} \begin{proof} Since for every solution $u\in\mathrm{MR}_{q,q_{\Gamma}}^{r}$
and $T>0$, $\mathfrak{a}_{u(s)}(u(s),u^{\infty})=0$, applying (\ref{eq:qlACP})
to $u-u^{\infty}$ shows the energy balance 
\begin{equation}
\Vert u(s)-u^{\infty}\Vert_{\mathrm{L}^{2,2}}^{2}+\int_{0}^{s}\mathfrak{a}_{u(t)}(u(t),u(t))\,\mathrm{d}t=\Vert u^{0}-u^{\infty}\Vert_{\mathrm{L}^{2,2}}^{2},\label{eq:L2 energy balance-1}
\end{equation}
for all $s>0$. By Lemma \ref{mp} and Assumption \ref{Akm}, 
\begin{align*}
\mathfrak{l}_{u(t)}(u(t),u(t)) & \geq C\Vert\nabla u(t)\Vert_{\mathrm{L}^{2,2}}^{2},\text{ and}\\
\mathfrak{m}_{u(t)}(u(t),u(t)) & \geq\underline{m}\left(\int_{\Gamma}(\tp-\tg)^{2}(t)+(\tm-\tg)^{2}(t)+(\tp-\tm)^{2}(t)\,\mathrm{d}\mathcal{H}_{d-1})\right).
\end{align*}
Hence, with the following Poincaré-type inequality, the claim follows
directly from Gronwall's inequality. \end{proof} 

\begin{lemma}\label{BIPI}(Bulk-Interface Poincaré Inequality) Let
$u\in\mathrm{W}^{1,2,2}$ and $u^{\infty}$ the equilibrium associated
to $u$. Then there is a constant $C>0$, independent of $u$, such
that 
\begin{equation}
\Vert u-u^{\infty}\Vert_{\mathrm{L}^{2,2}}^{2}\leq C(\Vert\nabla u\Vert_{\mathrm{L}^{2,2}}^{2}+\Vert\tp-\tg\Vert_{L^{2}(\Gamma)}^{2}+\Vert\tm-\tg\Vert_{L^{2}(\Gamma)}^{2}+\Vert\tp-\tm\Vert_{L^{2}(\Gamma)}^{2}).\label{eq:bipPoinc}
\end{equation}
\end{lemma}

\begin{proof} For any $u\in\mathrm{L}^{1,1}$, let in the following
$\ttp:=\frac{1}{|\op|}\int_{\op}\tp$, $\ttm:=\frac{1}{|\om|}\int_{\om}\tm$
and $\ttg:=\frac{1}{\vert\Gamma\vert}\int_{\Gamma}\tg$ and let $\bar{u}=(\ttp,\ttm,\ttg)\in\mathbb{R}^{3}$.
To prove (\ref{eq:bipPoinc}), we use the two (standard) versions
of Poincaré's inequality, see e.g. \citep[Theorem~1 and Corollary~3]{BouCha2007PI}.
For all $u_{+}\in W^{1,p}(\op)$, 
\begin{enumerate}
\item there is a constant $\bar{C}_{+}>0$, such that 
\begin{equation}
\Vert u_{+}-\bar{u}_{+}\Vert_{L^{2}(\op)}^{2}\leq\bar{C}_{+}\Vert\nabla u_{+}\Vert_{L^{2}(\op)}^{2},\quad\text{and,}\label{eq:poinAV}
\end{equation}
\item there is a constant $C_{+}^{\Gamma}>0$, such that 
\begin{equation}
\Vert u_{+}\Vert_{L^{2}(\op)}^{2}\leq C_{+}^{\Gamma}(\Vert\nabla u_{+}\Vert_{L^{2}(\op)}^{2}+\frac{1}{|\Gamma|}|\int_{\Gamma}u_{+}|^{2}).\label{eq:poinGamma}
\end{equation}
\end{enumerate}
Clearly, analogous statements hold for $\om$ with constants $\bar{C}_{-}>0$
and $C_{-}^{\Gamma}>0$ and \eqref{eq:poinAV} holds for $u_{\Gamma}$
on the manifold $\Gamma$ with constant $\bar{C}_{\Gamma}>0$. An
elementary calculation shows that 
\[
\Vert u-u^{\infty}\Vert_{\mathrm{L}^{2,2}}^{2}=\Vert u-\bar{u}\Vert_{\mathrm{L}^{2,2}}^{2}-V(u^{\infty})^{2}+|\op|\ttp^{2}+|\om|\ttm^{2}+|\Gamma|\ttg^{2}.
\]
Inserting $V\tinf=|\op|\ttp+|\om|\ttm+|\Gamma|\ttg$ 
gives 
\begin{equation}
\Vert u-\tinf\Vert_{\mathrm{L}^{2,2}}^{2}=\Vert u-\bar{u}\Vert_{\mathrm{L}^{2,2}}^{2}+\frac{|\op||\om|}{V}(\ttp-\ttm)^{2}+\frac{|\op||\Gamma|}{V}(\ttp-\ttg)^{2}+\frac{|\om||\Gamma|}{V}(\ttm-\ttg)^{2}.\label{eq:infEst1}
\end{equation}
By (\ref{eq:poinAV}), $\Vert u-\bar{u}\Vert_{\mathrm{L}^{2,2}}^{2}\leq(\bar{C}_{+}+\bar{C}_{-}+\bar{C}_{\Gamma})\Vert\nabla u\Vert_{\mathrm{L}^{2,2}}^{2}$,
so it remains to estimate the last three terms in (\ref{eq:infEst1})
by the right-hand-side in \eqref{eq:bipPoinc}. By Hölder's inequality
and by (\ref{eq:poinGamma}), 
\begin{align*}
(\ttp-\ttg)^{2} & =\frac{1}{|\op|^{2}}(\int_{\op}\tp-\ttg)^{2}\leq\frac{1}{|\op|}\Vert\tp-\ttg\Vert_{L^{2}(\op)}^{2}\\
 & \leq\frac{C_{+}^{\Gamma}}{|\op|}(\Vert\nabla\tp\Vert_{L^{2}(\op)}^{2}+\frac{1}{|\Gamma|}|\int_{\Gamma}\tp-\ttg|^{2})\\
 & \leq\frac{C_{+}^{\Gamma}}{|\op|}(\Vert\nabla\tp\Vert_{L^{2}(\op)}^{2}+\Vert\tp-\tg\Vert_{L^{2}(\Gamma)}^{2}).
\end{align*}
The term $(\ttm-\ttg)^{2}$ can be estimated analogously. In order
to estimate the last term $(\ttp-\ttm)^{2}$, insert $-\ttg+\ttg$
and use the previous estimates. 
With this strategy, it is clear that for (\ref{eq:bipPoinc}) to hold,
it is sufficient that two of the three coefficient functions $m_{+},m_{-},m_{\Gamma}$
are positive, so not every pair of unknowns needs to interact across
$\Gamma$. It is also sufficient for two of the coefficients to be
positive to guarantee the structure of the kernel of $\mathfrak{a}_{u}$
in \eqref{eq:kera}. This concludes the proof of Lemma \ref{BIPI}
and thus of Theorem \ref{thm:convToEq}. \end{proof}

In addition to exponential stability of $u^{\infty}$ within the sets
of initial data with equal mass, Theorem \ref{thm:convToEq} immediately
implies stability of $u^{\infty}$ in $\mathrm{X}_{q,q_{\Gamma}}^{r}$:
\begin{coro}\label{stab} For every $v^{\infty}\in\mathbb{R}_{+}$,
$\varepsilon>0$, if $u^{0}\in\mathrm{X}_{q,q_{\Gamma}}^{r}$ with
$\Vert u^{0}-v^{\infty}\Vert_{\mathrm{L}^{1,1}}<\varepsilon V$, then
$\vert u^{\infty}-v^{\infty}\vert<\varepsilon$. \end{coro} \begin{proof}
A direct calculation shows that 
\begin{align*}
\vert u^{\infty}-v^{\infty}\vert & =\frac{1}{V}\left|\int_{\op}u_{+}^{0}(x)-v^{\infty}\,\mathrm{d}x+\int_{\om}u_{-}^{0}(x)-v^{\infty}\,\mathrm{d}x+\int_{\Gamma}u_{\Gamma}^{0}(y)-v^{\infty}\,\mathrm{d}y\right|\\
 & \leq\frac{1}{V}\Vert u^{0}-v^{\infty}\Vert_{\mathrm{L}^{1,1}}.
\end{align*}
\end{proof}

\section{Onsager modeling, extensions and concluding remarks\label{sec:Extensions-and-concluding}}

\subsection{Entropic gradient structure for heat transfer (Onsager model)\label{sub:Onsager-Model}}

The system in \eqref{eq:plus}–\eqref{eq:Gamma}was motivated by non-equibirum
thermodynamical modeling of heat transfer and diffusion processes
across interfaces, \citep{KjeBed08NETD}, \citep{Otti05BET}, and
based on the results in \citep{GlitzkyMielke2013} and \citep{Miel13BII}.
For example, in \citep{Miel13BII}, it is shown that for \emph{flat}
interfaces $\Gamma$, the heat transfer \emph{Onsager} or \emph{gradient
system} associated to 
\begin{equation}
\dot{\theta}=\mathcal{K}(\theta)\mathrm{D}\mathcal{S}(\theta)\label{Onsager}
\end{equation}
is represented by the set of equations 
\begin{equation}
\left\{ \!\!\begin{array}{rcll}
\dot{\theta}_{\pm}\!+\!\frac{1}{c_{\pm}}\mathrm{div}(K_{\pm}(\theta_{\pm})\nabla\frac{1}{\theta_{\pm}})\!\!\! & = & \!\!\!0, & \!\!\text{in }(0,T)\times\Omega_{\pm},\\
(\frac{K_{\pm}(\theta_{\pm})}{c_{\pm}}\nabla\frac{1}{\theta_{\pm}})\nu_{\pm}\!+\!M_{\pm}(\theta)(\frac{1}{\theta_{\pm}}\!-\!\frac{1}{\theta_{\Gamma}})\!+\!M_{\Gamma}(\theta)(\frac{1}{\theta_{\pm}}\!-\!\frac{1}{\theta_{\mp}})\!\!\! & = & \!\!\!0, & \!\!\text{on }(0,T)\times\Gamma,\\
(K_{\pm}(\theta_{\pm})\nabla\frac{1}{\theta_{\pm}})\nu_{\pm}\!\!\! & = & \!\!\!0, & \!\!\text{on }(0,T)\times\{\partial\Omega_{\pm}\backslash\Gamma\},
\end{array}\right.\label{eq:pmOS}
\end{equation}
on the bulk parts, and 
\begin{equation}
\left\{ \!\!\begin{array}{rcll}
\dot{\theta}_{\Gamma}\!+\!\frac{1}{c_{\Gamma}}\div(K_{\Gamma}(\theta)\nabla\frac{1}{\theta_{\Gamma}})\!-\!M_{+}(\theta)(\frac{1}{\theta_{+}}\!-\!\frac{1}{\theta_{\Gamma}})\!-\!M_{-}(\theta)(\frac{1}{\theta_{-}}\!-\!\frac{1}{\theta_{\Gamma}})\!\!\! & = & \!\!\!0, & \!\!\text{in }(0,T)\times\Gamma,\\
(K_{\Gamma}(\theta)\nabla\frac{1}{\theta_{\Gamma}})\nu_{\Gamma}\!\!\! & = & \!\!\!0, & \!\!\text{on }(0,T)\times\partial\Gamma,
\end{array}\right.\label{eq:gammaOS}
\end{equation}
on the flat interface $\Gamma$, where $c_{\pm},c_{\Gamma}>0$ are
the specific heats of bulk and interface materials, respectively,
and the coefficients $K,M$ specify thermal conductivity within materials
and across $\Gamma$ in an entropic modelling. In \eqref{Onsager},
$\mathcal{S}$ is the total entropy functional 
\[
\mathcal{S}(\theta)=\int_{\Omega_{+}}c_{+}\log\theta_{+}\,\mathrm{d}x+\int_{\Omega_{-}}c_{-}\log\theta_{-}\,\mathrm{d}x+\int_{\Gamma}c_{\Gamma}\log\theta_{\Gamma}\,\mathrm{d}y,
\]
and $\mathcal{K}$ is the Onsager operator corresponding to the the
dual dissipation potential 
\begin{align}
2\Psi^{*}(\theta,\phi) & =2\Psi_{+}^{*}(\theta_{+},\phi_{+})+2\Psi_{-}^{*}(\theta_{-},\phi_{-})+2\Psi_{\Gamma}^{*}(\tr\theta,\tr\phi)\nonumber \\
 & =\int_{\Omega_{+}}\nabla\frac{\phi_{+}}{c_{+}}\cdot K_{+}(\theta_{+})\nabla\frac{\phi_{+}}{c_{+}}\,\mathrm{d}x+\int_{\Omega_{-}}\nabla\frac{\phi_{-}}{c_{-}}\cdot K_{-}(\theta_{-})\nabla\frac{\phi_{-}}{c_{-}}\,\mathrm{d}x\nonumber \\
 & +\int_{\Gamma}\nabla_{\Gamma}\frac{\phi_{\Gamma}}{c_{\Gamma}}\cdot K_{\Gamma}(\mathrm{tr}_{\Gamma}\,\theta)\nabla\frac{\phi_{\Gamma}}{c_{\Gamma}}\,\mathrm{d}y+\int_{\Gamma}M_{\Gamma}(\tr\theta)(\frac{\tr\phi_{+}}{\tr c_{+}}-\frac{\tr\phi_{-}}{\tr c_{-}})^{2}\,\mathrm{d}y\nonumber \\
 & +\int_{\Gamma}M_{+}(\tr\theta)(\frac{\tr\phi_{+}}{\tr c_{+}}-\frac{\tr\phi_{\Gamma}}{\tr c_{\Gamma}})^{2}+M_{-}(\tr\theta)(\frac{\tr\phi_{-}}{\tr c_{-}}-\frac{\tr\phi_{\Gamma}}{\tr c_{\Gamma}})^{2}\,\mathrm{d}y.\label{eq:Psi}
\end{align}
Equations \eqref{eq:pmOS} and \eqref{eq:gammaOS} are equivalent
to \eqref{eq:plus}–\eqref{eq:Gamma} by differentiating $\nabla\frac{1}{\theta}$
to $-\frac{1}{\theta^{2}}\nabla\theta$ and writing $\frac{1}{\theta_{\Gamma}\theta_{+}}(\theta_{+}-\theta_{\Gamma})$
instead of $(\frac{1}{\theta_{\Gamma}}-\frac{1}{\theta_{+}})$, for
every term of this kind. The coefficients $K$ and $k$ and $M$ and
$m$ are then related via $m_{\pm}(\tr\theta)=\frac{M_{\pm}(\tr\theta)}{\theta_{\Gamma}\tr\theta_{\pm}}$,
$m_{\Gamma}(\tr\theta)=\frac{M_{\Gamma}(\tr\theta)}{\tr\theta_{+}\tr\theta_{-}}$,
$k_{\pm}(\theta_{\pm})=\frac{K_{\pm}(\theta_{\pm})}{\theta_{\pm}^{2}}$
and $k_{\Gamma}(\tr\theta)=\frac{K_{\Gamma}(\tr\theta)}{\theta_{\Gamma}^{2}}$.

We check the applicability and the implications of Theorem \ref{thm:local}:
It is straightforward to see that $K,M$ satisfy Assumption \ref{Akm}
if and only if $k,m$ satisfy Assumption \ref{Akm}. So if Assumption
\ref{Akm} on $K,M$ is respected in an entropic modeling, well-posedness
and exponential stability are obtained. In particular, the positivity
of two components of $M$ guarantees entropy production of the bulk-interface
interaction and information on the Onsager system given by $\mathcal{S}$
and $\Psi^{*}$ is retrieved: Starting from positive intial values,
$l>0$, the regularity in Theorem \ref{thm:local} and the maximum
principle justify rigorously the equivalence of \eqref{eq:pmOS},
\eqref{eq:gammaOS} and \eqref{eq:plus}–\eqref{eq:Gamma} and the
solution provides the gradient flow of $\mathcal{S}$ with respect
to the dual dissipation metric $\Psi^{*}$. The entropy $\mathcal{S}(\theta(t))$
is well-defined along orbits and $-\mathcal{S}$ provides a strict
Lyapunov functional by the energy balance $-\frac{\mathrm{d}}{\mathrm{d}t}\mathcal{S}(\theta(t))+2\Psi^{*}\big{(}\theta(t),\frac{c}{\theta(t)}\big{)}=0$
and the fact that $2\Psi^{*}\big(\theta(t),\frac{c}{\theta(t)}\big)=0$
implies $\mathfrak{a}_{\theta(t)}(\theta(t),\theta(t))=0$ along the
positive orbits of $\theta$. By Theorem \ref{thm:convToEq}, exponential
stability holds in the sense that $\Vert c\theta(t)-c\theta^{\infty}\Vert_{\mathrm{L}^{2,2}}\leq e^{-\delta t}\Vert c\theta^{0}-c\theta^{\infty}\Vert_{\mathrm{L}^{2,2}}$
for some $\delta>0$.

\subsection{Small extensions and further remarks\label{sub:Other-remarks}}

The next remarks concern extensions of Theorem \ref{thm:local},
mostly based on perturbation theory for maximal parabolic regularity.

\begin{remark} Clearly, the analysis above includes the simpler case
of bulk-surface interaction with $\om=\emptyset$, without the variable
$u_{-}$ and with $m_{-}=m_{\Gamma}=0$. \end{remark}

\begin{remark} If the Lipschitz dependence of $k$, $m$ and $\mathcal{F}$
on $u$ in Assumptions \ref{Akm}\eqref{akm3} and \ref{Lipschitz}
is improved to $C^{n}$, $n\in\mathbb{N}\cup\{\infty,\omega\}$, then
the solution $u$ in Theorem \ref{thm:local} gains time regularity
by \citep[Theorem~ 5.1]{PruessBari2002}, i.e. it follows that 
\[
u\in C^{n}(J_{T};X_{r,q,q_{\Gamma}})\cap C^{n+1-1/r}(J_{T};\mathrm{W}_0^{-1,q,q_{\Gamma}})\cap C^{n-1/r}(J_{T};\mathrm{W}^{1,q,q_{\Gamma}})
\]
and that $u\in C^{\infty}(J_{T};\mathrm{W}^{1,q,q_{\Gamma}})$ if
$n=\infty$ and $u$ is real analytic on $J_{T}$ if $n=\omega$.
\end{remark}

\begin{remark} The coefficient functions $k,m$ and external forces
and inhomogeneous boundary conditions $f,g,h$ may additionally depend
on time. For example, Theorems \ref{thm:local} and \ref{thm:convToEq}
continue to hold if Assumption \ref{Akm} holds uniformly in $t\in(0,\infty)$
for $k,m$ and $t\mapsto\mathcal{A}_{u(t)}\in\mathcal{L}(\mathrm{W}^{1,q,q_{\Gamma}},\mathrm{W}_0^{-1,q,q_{\Gamma}})$
is continuous for all $u\in\mathrm{X}_{q,q_{\Gamma}}^{r}$ and if
Assumption \ref{Lipschitz} holds where $t\mapsto\mathcal{F}(t,u)$
is in $L^{r}(J_{T};\mathrm{W}_0^{-1,q,q_{\Gamma}})$, cf. \citep[Section~3]{PruessBari2002}.
\end{remark}

\begin{remark} The results in Theorem \ref{thm:local} extend to
perturbations of $\mathcal{A}_{q}$ by lower-order terms like transport
terms $b\cdot\nabla u_{\pm}$, $b\in\mathbb{R}^{d}$. In particular,
with suitable regularity assumptions, the coefficients $c_{\pm}:\Omega_{\pm}\to\mathbb{R}_{+}\backslash\{0\}$
and $c_{\Gamma}:\Gamma\to\mathbb{R}_{+}\backslash\{0\}$ in Subsection
\ref{sub:Onsager-Model} can be chosen to depend on the spatial variables.
\end{remark}

\begin{remark} An exponential convergence rate as in \eqref{eq:infEst1}
also holds in $\mathrm{L}^{p,p}$ for $p>2$ due to interpolation
of the $p$-norms since 
\[
\Vert u(s)-u^{\infty}\Vert_{\mathrm{L}^{\infty,\infty}}\leq3\max(\vert l-u^{\infty}\vert,\vert L-u^{\infty}\vert)=:C_{\infty}
\]
 is bounded uniformly in time. For all $s\geq0$, with $\theta=1-\frac{2}{p}$,
we get
\begin{align*}
\Vert u(s)-u^{\infty}\Vert_{\mathrm{L}^{p,p}} & \leq C_{\infty}^{\theta}\Vert u(s)-u^{\infty}\Vert_{\mathrm{L}^{2,2}}^{1-\theta}\\
 & \leq C_{\infty}^{\theta}e^{-\frac{2\delta}{p}s}\Vert u^{0}-u^{\infty}\Vert_{\mathrm{L}^{2,2}}^{\frac{2}{p}}\\
 & \leq e^{-\frac{2\delta}{p}s}\Vert u^{0}-u^{\infty}\Vert_{\mathrm{L}^{p,p}}.
\end{align*}
Additional rates that exploit the gradient structure of system \eqref{eq:plus}
– \eqref{eq:Gamma} depend on the choice of $k,m$ or of the energy
and dissipation functionals $\mathcal{S}$ and $\Psi$, see e.g. \citep{BFL18vsrd}
for exponential $\mathrm{L}^{1}$-rates for volume-surface reaction-diffusion
based on entropy production. \end{remark}

\subsection*{Acknowledgements}

K.D. was supported by the European Research Council via ``ERC-2010-AdG
no. 267802 (Analysis of Multiscale Systems Driven by Functionals)''.


\bibliographystyle{abbrv}

\end{document}